\documentclass[article,noinfoline]{imsart}

\usepackage{graphicx,psfrag,epsf}
\usepackage{enumerate}
\usepackage{natbib}
\usepackage{url} % not crucial - just used below for the URL 

\usepackage{amssymb,amsmath,amsthm,dsfont,tikz,ifthen,subfigure,booktabs,slashbox
}
\usepackage{titlesec,mathptmx}
\usepackage[small,bf]{caption}
\usepackage[T1]{fontenc}
\usepackage[latin1]{inputenc}
\usepackage[ngerman,english]{babel}
\usepackage[babel,german=quotes]{csquotes}
\usepackage{wasysym}
\usepackage[weather]{ifsym}
\usepackage{multirow}
\usepackage{amsbsy}
\usepackage{multibib}
\usepackage{scalerel}
\usepackage{enumitem}   
\usepackage{listings}
\usepackage{subfigure}
\usetikzlibrary{plotmarks,calc,arrows}
\usepackage{rotating}
\usepackage{units,nicefrac}

\usepackage{amsmath,amsthm}
\usepackage{hyperref}
\hypersetup{
	pdfstartview={FitH},
	colorlinks=true,
	linkcolor=black,
	urlcolor=black,
	citecolor=black,
	bookmarksopen,
	bookmarksopenlevel={1},
	bookmarksnumbered
}

\newtheorem{satz}{Theorem}[section]

\newtheorem{lemma}[satz]{Lemma}

\theoremstyle{definition}
\newtheorem{kor}[satz]{Corollary}

\theoremstyle{remark}

\theoremstyle{plain}

\providecommand{\smalltitle}[1]{
	\subsubsection*{#1}
}
\providecommand{\smalltitle}[1]{
	\subsection*{#1}
}

\numberwithin{equation}{section}

\usepackage{marginnote}

\newcommand\R{\mathbb R} % K?rper der reellen Zahlen
\newcommand\N{\mathbb N} % nat?rliche Zahlen	
\newcommand\C{\mathbb C} % K?rper der komplexen Zahlen
\newcommand\E{\mathbb E} % Erwartungswert
\newcommand{\X}{\mathcal X}
\newcommand\A{\mathcal{A}} %Sigma Algebra A
\newcommand\B{\mathcal{B}} %Sigma Algebra B
\newcommand\norm[1]{||\,#1\,||_2} % 2-Norm
\newcommand\norminf[1]{||\,#1\,||_\infty} % Infinity Norm
\newcommand\inv{^{-1}} % inverse
\renewcommand\d{\mathrm{d}  }
\newcommand\dx{\mathrm{d} x }

\newcommand{\IGNORE}[1]{}

\newcommand\pvec[1]{\mathbf{#1}}

\newcommand{\bbauf}{\big(}
\newcommand{\bbzu}{\big)}

\newcommand{\babs}{\big|}
\newcommand{\Babs}{\Big|}

\newcommand{\Pxntheta}{P_{\Xntheta}}
\newcommand{\Pxtheta}{P_{\Xtheta}}
\newcommand{\Pyntheta}{P_{\Yntheta}}

\newcommand\unifweakTheta{\stackrel{w,\Theta}{\Longrightarrow}}
\newcommand\unifdistTheta{\stackrel{D,\Theta}{\Longrightarrow}}
\newcommand{\muntheta}{\mu_{n}^\vartheta}
\newcommand{\mutheta}{\mu^\vartheta}
\newcommand{\dmun}{\mathrm{d}\muntheta}
\newcommand{\dmu}{\mathrm{d}\mutheta}

\newcommand{\Xntheta}{X_{n}^\vartheta}
\newcommand{\Yntheta}{Y_{n}^\vartheta}
\newcommand{\FXntheta}{F_{\Xntheta}}
\newcommand{\FXtheta}{F_{\Xtheta}}
\newcommand{\Xtheta}{X^\vartheta}

\begin{document}

\begin{frontmatter}
	
	\title{Uniform approximation in classical weak convergence theory}
	\runtitle{Uniform weak convergence theory}
	
	\begin{aug}
		\author{\fnms{Viktor}  \snm{Bengs}\thanksref{t1}\ead[label=e1]{bengs@mathematik.uni-marburg.de}}
		\and
		\author{\fnms{Hajo} \snm{Holzmann}\thanksref{t1,t2}\ead[label=e2]{holzmann@mathematik.uni-marburg.de}}
		%			\ead[label=u1,url]{http://www.foo.com}

		\address{Department of Mathematics and Computer Science. \\ Philipps-Universit\"at Marburg.	\\Hans-Meerwein-Stra\ss e, 35032 Marburg, Germany   \\
			\printead{e1,e2}}
		
		\thankstext{t1}{Supported by DFG Grant Ho 3260/5-1.}
		\thankstext{t2}{Corresponding author.}
		\runauthor{V. Bengs and H. Holzmann}
		
		\affiliation{ Philipps-Universit\"at Marburg.}
		
	\end{aug}
	
	\begin{abstract} \
		A common statistical task lies in showing asymptotic normality of certain statistics. 
		In many of these situations, classical textbook results on weak convergence theory suffice for the problem at hand.
		However, there are quite some scenarios where stronger results are needed in order to establish an asymptotic normal approximation uniformly over a family of probability measures.
		In this note we collect some results in this direction. We restrict ourselves to weak convergence in $\R^d$ with continuous limit measures. 
	\end{abstract}
	
	\begin{keyword}
		\kwd{Central limit theorem}
		\kwd{Continuous mapping theorem}
		\kwd{Cram\'er-Wold theorem}
		\kwd{L\'evy's continuity theorem}
		\kwd{Lindeberg-Feller theorem}
		\kwd{Portmanteau theorem}
		\kwd{Slutzky's lemma}
	\end{keyword}
	
\end{frontmatter}

%Motivation needed?

In this note we collect some results on asymptotic normal approximation uniformly over a family of probability measures. We restrict ourselves to weak convergence in $\R^d$ with continuous limit measures. 

In this paper we derive a range of uniform versions of  results known in the classical weak convergence theory as for instance in \citet{durrett2010probability, van2000asymptotic, pollard2012convergence}.

Section \ref{sec_definitions} contains the definitions of what we shall mean by uniform weak convergence and uniform convergence in distribution. 
More specifically, we establish in Section \ref{sec_portmanteau} a uniform version of Portmanteau's Theorem for the defined uniform weak convergence and show in Section \ref{sec_weak_on_generator} the sufficiency of deriving weak convergence on generators which are closed under finite intersections.
In Section \ref{sec_relation} the relationship between uniform weak convergence of probability measures of random variables and uniform convergence in distribution of their distribution functions is established.
In what follows, uniform versions of L\'evy's Continuity Theorem and the Cram\'er-Wold Theorem are derived in Section \ref{sec_levy_cramer_wold} and uniform versions of the Continuous Mapping Theorem and Slutzky's Theorem are established in Section \ref{sec_cmt_slutzky}.
Finally, a uniform version of the Lindeberg-Feller Theorem is verified in Section \ref{sec_lindeberg_feller} providing a tool for obtaining uniform central limit theorems.

%\section{Notation}

\section{Preliminaries} \label{sec_definitions}

\subsection{Notation}
We use the following notation in this paper.
By $F_X$ we denote the cumulative distribution function of a random vector $X$ and by $P_X$ its law.
Let $\phi_X$ be the characteristic function of a random vector $X.$

Moreover, we assume that $\Theta$ is some arbitrary set and for any $\vartheta \in \Theta,$ $(\Xntheta)_{n\in\N}$   is a sequence of real-valued random vectors in $\R^d$.
Likewise, for any $\vartheta \in \Theta,$ let $\Xtheta$ be random vectors in $\R^d$ with \emph{continuous} distribution.

\subsection{Uniform convergence in distribution}

We write $ \Xntheta \unifdistTheta \Xtheta  $ if 
$$	 \sup_{\vartheta\in \Theta} |	\FXntheta(\pvec{x}) - \FXtheta(\pvec{x})		|=o(1), \quad \forall \pvec{x} \in \R^d.	$$
In this case we say that $\Xntheta$ \emph{converges uniformly over} $\Theta$ \emph{in distribution to} $\Xtheta.$
%
%where $C_{F^\Theta}$ is the intersection of all continuity points of $\FXtheta$ over $\vartheta \in \Theta.$
%
We say that $(\Pxtheta)_{\vartheta \in \Theta}$ is \emph{ uniformly absolutely continuous over $\Theta$ with respect to some continuous probability measure $Q$}, if 
for any $\varepsilon>0$ there exists a $\delta>0$ such that for any measurable $A\subset \R^d$ with $Q(A)<\delta$ one has that $ \sup_{  \vartheta \in \Theta } \Pxtheta(A)<\varepsilon.$ 
Note that by continuous probability measure we mean that the measure of singletons is zero, i.e.\ $Q(\{x\})=0$ for any $x\in \R^d$.

\subsection{Uniform weak convergence}

In the same spirit we define uniform weak convergence for probability measures.
Let $(\X,d)$ be some metric space and let $\A$ be its Borel $\sigma$-algebra.
Let for any $\vartheta \in \Theta,$ $(\muntheta)_{n\in\N}$   be a sequence of probability measures on $(\X,\A).$
Similarly, for any $\vartheta \in \Theta,$ let $\mutheta$ be a probability measure on $(\X,\A).$
In the same manner as for the law of random vectors we define uniform absolute continuity over $\Theta$ for $(\mutheta)_{\vartheta\in\Theta},$ that is $(\mutheta)_{\vartheta\in\Theta}$ is \emph{uniformly absolutely continuous over $\Theta$ with respect to some continuous probability measure $\mu$} if 
for any $\varepsilon>0$ there exists a $\delta>0$ such that for any $A\in  \A$ with $\mu(A)<\delta$ it follows that $ \sup_{  \vartheta \in \Theta } \mutheta(A)<\varepsilon.$
Eventually, we say that $\muntheta$ \emph{converges uniformly weakly over} $\Theta$ to $\mutheta$ and write $ \muntheta \unifweakTheta \mutheta		$ if and only if
$$	\sup_{\vartheta \in \Theta}	\Babs	\int g \, \dmun - \int g \,\dmu		\Babs =o(1)	$$
for any real-valued, bounded and continuous function $g:\X \to \R.$
%

%%%%%%%%%%%%%%%%%%%%%%%%%%%%%%%%%%%%%%%%%%%%%%%%%%%%%%%%%%%%%%%%%%%%%%%%%%%%%%%%%%%%%%%%%%%%%%%%%%%%%%%%%%%
%%%%%%%%%%%%%%%%%%%%%%%%%%%%%%%%%%%%%%%%%%%%%%%%%%%%%%%%%%%%%%%%%%%%%%%%%%%%%%%%%%%%%%%%%%%%%%%%%%%%%%%%%%%
%%%%%%%%%%%%%%%%%%%%%%%%%%%%%%%%%%%%%%%%%%%%%%%%%%%%%%%%%%%%%%%%%%%%%%%%%%%%%%%%%%%%%%%%%%%%%%%%%%%%%%%%%%%
%%%%%%%%%%%%%%%%%%%%%%%%%%%%%%%%%%%%%%%%%%%%%%%%%%%%%%%%%%%%%%%%%%%%%%%%%%%%%%%%%%%%%%%%%%%%%%%%%%%%%%%%%%%
%		Uniform Portmanteau
%%%%%%%%%%%%%%%%%%%%%%%%%%%%%%%%%%%%%%%%%%%%%%%%%%%%%%%%%%%%%%%%%%%%%%%%%%%%%%%%%%%%%%%%%%%%%%%%%%%%%%%%%%%
%%%%%%%%%%%%%%%%%%%%%%%%%%%%%%%%%%%%%%%%%%%%%%%%%%%%%%%%%%%%%%%%%%%%%%%%%%%%%%%%%%%%%%%%%%%%%%%%%%%%%%%%%%%
%%%%%%%%%%%%%%%%%%%%%%%%%%%%%%%%%%%%%%%%%%%%%%%%%%%%%%%%%%%%%%%%%%%%%%%%%%%%%%%%%%%%%%%%%%%%%%%%%%%%%%%%%%%
%%%%%%%%%%%%%%%%%%%%%%%%%%%%%%%%%%%%%%%%%%%%%%%%%%%%%%%%%%%%%%%%%%%%%%%%%%%%%%%%%%%%%%%%%%%%%%%%%%%%%%%%%%%

\section{Uniform Portmanteau Theorem} \label{sec_portmanteau}

We start with a uniform version of the Portmanteau theorem.

%The following uniform version of  Portmanteau's Theorem will be of great assistance as well for many proofs in the following.
%
\begin{satz} [Uniform Portmanteau] \label{lemma:uniform_portmanteau}
	Let $(\mutheta)_{\vartheta\in\Theta}$ be  uniformly absolutely continuous over $\Theta$ with respect to some continuous probability measure $\mu.$ 
	The following conditions are equivalent:
	\begin{enumerate}
		\item $ \muntheta \unifweakTheta \mutheta;$
		\item $	\sup_{\vartheta \in \Theta}	|	\int g \, \dmun - \int g \,\dmu		|=o(1)$ 
		for any real-valued, bounded and Lipschitz continuous function $g:\X \to \R;$
		\item $	\limsup_n 	\sup_{\vartheta \in \Theta} (\muntheta(F) - \mutheta(F))	  \leq 0		$ for any closed set $F\subset \X;$
		\item $	\liminf_n 	\inf_{\vartheta \in \Theta} (\muntheta(G) - \mutheta(G))	  \geq 0		$ for any open set $G\subset \X;$
		\item $	\limsup_n 	\sup_{\vartheta \in \Theta} \bbauf g \, \dmun -  \int g \, \dmu)	  \leq 0		$ for any real-valued, bounded and upper-semi-continuous function $g:\X \to \R;$
		\item $	\liminf_n 	\inf_{\vartheta \in \Theta} \bbauf g \, \dmun -  \int g \, \dmu)	  \geq 0		$ for any real-valued, bounded and lower-semi-continuous function $g:\X \to \R;$
		\item $	\sup_{\vartheta \in \Theta}	|	\muntheta(A)- \mutheta(A)	|=o(1)		$ for any set $A\in \A$ such that $\sup_{\vartheta\in \Theta} \mutheta(\partial A)=0.$
	\end{enumerate}
\end{satz}

\allowdisplaybreaks
%
%\section{Proof of Lemma \ref{lemma:uniform_portmanteau}}
%
\begin{proof} [\sl Proof of Theorem \ref{lemma:uniform_portmanteau}]
	It is clear, that 1.\ implies 2.
	Now, we show that 2.\ implies 3.
	Let $F \subset \X$ be a closed set and for $k\in \N$ let
	$$	F_k	 = \{	x \in \X \ | \ d(x,F)\leq 1/k	\},	$$
	where $d(x,F)=\inf\{d(x,y) \ | \ y\in F \}.$
	Additionally, for any $x\in \X$ let $g_{k,F}(x)=\max\{1- k d(x,F),0	\}.$
	Then $g_{k,F}$ is Lipschitz continuous with $1_F \leq g_{k,F} \leq 1_{F_k}.$
	Thus,
	\begin{align*}
	\limsup_n 	\sup_{\vartheta \in \Theta} (\muntheta(F) - \mutheta(F))
	&= \limsup_n 	\sup_{\vartheta \in \Theta} \bbauf \int 1_F \dmun - \int 1_F \dmu\bbzu \\
	&\leq \limsup_n 	\sup_{\vartheta \in \Theta} \bbauf \int g_{k,F} \dmun - \int 1_F \dmu\bbzu \\
	&\leq \limsup_n 	\sup_{\vartheta \in \Theta} \bbauf \int g_{k,F} \dmun - \int  g_{k,F}  \dmu\bbzu \\
	&\quad + \limsup_n 	\sup_{\vartheta \in \Theta} \bbauf \int g_{k,F} \dmu - \int 1_{F_k}  \dmu\bbzu \\
	&\quad  + 	\sup_{\vartheta \in \Theta} \bbauf  \int 1_{F_k} \dmu - \int 1_{F}  \dmu\bbzu.
	\end{align*}
	Now, the first term on the right-hand side of the latter inequality is zero by assumption 2.
	The second is smaller than zero by construction of $g_{k,F}.$
	Concerning the third term, note that by continuity of measure $\mu(F_k\backslash F)\to 0$ for $k\to \infty.$
	Consequently, considering the limit $k \to \infty$ it follows by the  uniform absolute continuity of $(\mutheta)_{\vartheta\in\Theta}$ with respect to $\mu$ and reasons of monotonicity that the third term tends to zero.
	
	3.\ is equivalent to 4.\ by taking complements.
	Similarly, 5.\ and 6.\ are apparently equivalent.
	Moreover, 5.\ and 6.\ together imply 1.
	%\\
	We prove that 3.\ implies 5. 
	Let $g:\X \to \R$ be a bounded, upper-semi-continuous function with $a<g<b$ for some $a,b\in \R.$
	%	
	%	Without loss of generality assume that $a,b>0.$
	%
	Then, 
	\begin{align*}
	\sup_{\vartheta \in \Theta} \bbauf \int g \, \dmun -  \int g \, \dmu\bbzu	
	&= \sup_{\vartheta \in \Theta} \int_a^b  \bbauf\muntheta(g\geq x)  -\mutheta(g\geq x) \bbzu \, \dx   	\\
	&\leq  \int_a^b \sup_{\vartheta \in \Theta} \bbauf\muntheta(g\geq x)  -\mutheta(g\geq x) \bbzu \, \dx   
	\end{align*}
	Thus, by Fatou's lemma and by 3.
	\begin{align*}
	\limsup_n 	\sup_{\vartheta \in \Theta} \bbauf \int g \, \dmun -  \int g \, \dmu\bbzu   	
	&\leq \int_a^b \limsup_n \sup_{\vartheta \in \Theta} \bbauf\muntheta(g\geq x)  -\mutheta(g\geq x) \bbzu \, \dx \leq 0,
	\end{align*}
	which implies 5. It remains to incorporate 7.\ in the implication flow.
	Next, we show that 3.\ and 4.\ imply 7.\ and vice versa.
	Starting with 3.\ and 4.\ implies 7.\,, we consider some $A\in \A$ such that $\sup_{\vartheta \in \Theta} \mutheta (\partial A)=0.$
	Let $A^\circ$ denote its interior and $\bar A$  its closure.
	Clearly, $A^\circ \subset A \subset \bar A$ and $A^\circ$ is open and $\bar A$ is closed.
	Furthermore, $ \mutheta (A^\circ)= \mutheta (\bar A ) = \mutheta (A ) $ for any $\vartheta \in \Theta.$
	Thus, by 3.\ 
	\begin{align*}
	\limsup_n 	\sup_{\vartheta \in \Theta} (\muntheta(A) - \mutheta(A)) 
	\leq \limsup_n 	\sup_{\vartheta \in \Theta} (\muntheta(\bar A) - \mutheta(\bar A)) \leq 0
	\end{align*}
	and by 4.\
	\begin{align*}
	\liminf_n 	\inf_{\vartheta \in \Theta} (\muntheta(A) - \mutheta(A))
	\geq \liminf_n 	\inf_{\vartheta \in \Theta} (\muntheta(A^\circ) - \mutheta(A^\circ)) \geq 0.
	\end{align*}
	Both latter inequalities together imply 7.
	Now suppose 7.\ holds.
	Let $x\in \X$ and $F \subset \X$ be closed.
	Define for $r\geq 0$
	$$	B_F(r)= \{	x\in \X \ | \ d(x,F) \leq r 		\}, \quad \mbox{and} \quad C_F(r)=\{  x\in \X \ | \ d(x,F)=r	\}.  $$
	Then, $(C_F(r))_{r\geq 0}$ is a partition of $\X.$ 
	Note that there exists a countable set $R \subset [0,\infty)$ such that $\mu(C_F(r))=0$ for $r\in [0,\infty)\backslash R,$ otherwise we could contradict the measure properties of $\mu.$
	It holds that $\partial B_F(r) \subset C_F(r)$ and thus $\mu(\partial B_F(r) )=0$ for $r\in [0,\infty)\backslash R.$
	Due to uniform absolute continuity of $(\mutheta)_{\vartheta\in\Theta}$  with respect to $\mu$ one has that $\sup_{\vartheta \in \Theta} \mutheta(\partial B_F(r))=0$ for $r\in [0,\infty)\backslash R.$ 
	Hence, there exists a sequence $r_k\searrow 0$ such that 
	$$\sup_{\vartheta \in \Theta} \mutheta(\partial B_F(r_k))=0, \quad \forall k \in \N.$$
	Since $F \subset B_F(r_k)$ and by 7. we have for any $k\in \N$
	\begin{align*}
	\limsup_n 	\sup_{\vartheta \in \Theta} (\muntheta(F) - \mutheta(B_F(r_k)))
	\leq	\limsup_n 	\sup_{\vartheta \in \Theta} (\muntheta(B_F(r_k) - \mutheta(B_F(r_k)))=0.
	\end{align*}
	Considering the limit $k \to \infty$ and noticing that by reasons of monotonicity the left-hand side of the latter display tends to $\limsup_n 	\sup_{\vartheta \in \Theta} (\muntheta(F) - \mutheta(F) )  $ concludes the lemma.
\end{proof}

%%%%%%%%%%%%%%%%%%%%%%%%%%%%%%%%%%%%%%%%%%%%%%%%%%%%%%%%%%%%%%%%%%%%%%%%%%%%%%%%%%%%%%%%%%%%%%%%%%%%%%%%%%%
%%%%%%%%%%%%%%%%%%%%%%%%%%%%%%%%%%%%%%%%%%%%%%%%%%%%%%%%%%%%%%%%%%%%%%%%%%%%%%%%%%%%%%%%%%%%%%%%%%%%%%%%%%%
%%%%%%%%%%%%%%%%%%%%%%%%%%%%%%%%%%%%%%%%%%%%%%%%%%%%%%%%%%%%%%%%%%%%%%%%%%%%%%%%%%%%%%%%%%%%%%%%%%%%%%%%%%%
%%%%%%%%%%%%%%%%%%%%%%%%%%%%%%%%%%%%%%%%%%%%%%%%%%%%%%%%%%%%%%%%%%%%%%%%%%%%%%%%%%%%%%%%%%%%%%%%%%%%%%%%%%%
%	pi-system generators
%%%%%%%%%%%%%%%%%%%%%%%%%%%%%%%%%%%%%%%%%%%%%%%%%%%%%%%%%%%%%%%%%%%%%%%%%%%%%%%%%%%%%%%%%%%%%%%%%%%%%%%%%%%
%%%%%%%%%%%%%%%%%%%%%%%%%%%%%%%%%%%%%%%%%%%%%%%%%%%%%%%%%%%%%%%%%%%%%%%%%%%%%%%%%%%%%%%%%%%%%%%%%%%%%%%%%%%
%%%%%%%%%%%%%%%%%%%%%%%%%%%%%%%%%%%%%%%%%%%%%%%%%%%%%%%%%%%%%%%%%%%%%%%%%%%%%%%%%%%%%%%%%%%%%%%%%%%%%%%%%%%
%%%%%%%%%%%%%%%%%%%%%%%%%%%%%%%%%%%%%%%%%%%%%%%%%%%%%%%%%%%%%%%%%%%%%%%%%%%%%%%%%%%%%%%%%%%%%%%%%%%%%%%%%%%

%\subsubsection{Proof of  Theorem \ref{theorem_equivalence_weak_distrib_conv}}

\section{Uniform weak convergence on generators} \label{sec_weak_on_generator}

The following lemma shows that it suffices to show uniform weak convergence on some generator which is closed under finite intersections.

\begin{lemma} \label{lemma_generator_equivalence}
	Let $(\mutheta)_{\vartheta\in\Theta}$ be  uniformly absolutely continuous over $\Theta$ with respect to some continuous probability measure $\mu$  and let $\tilde \B$ be a collection of open subsets which is closed under finite intersections. 
	Furthermore, each open set of $\X$ can be represented as a countable union of elements in $\tilde \B.$
	Then, 
	\begin{align} \label{defi_unif_conv_on_generator}
	\sup_{\vartheta \in \Theta}	|	\muntheta(A)  -\mutheta(A)		| =o(1)
	\end{align}
	for all $A\in \tilde \B$ implies that	$ \muntheta \unifweakTheta \mutheta.$ 
\end{lemma}
\begin{proof}[\sl Proof of Lemma \ref{lemma_generator_equivalence}]
	Let $G_1,G_2 \in \tilde \B.$ 
	By (\ref{defi_unif_conv_on_generator}) and since $G_1\cap G_2 \in \tilde \B,$ we have
	\small
	\begin{align*}
	\sup_{\vartheta \in \Theta}	&|	\muntheta(I_1 \cup I_2)  -\mutheta(I_1 \cup I_2)		| \\
	&\leq  	\sup_{\vartheta \in \Theta} \Big[	\babs	\muntheta(I_1 )  -\mutheta(I_1)		\babs + 	\babs	\muntheta(I_2 )  -\mutheta(I_2)		\babs  + 	\babs	\muntheta(I_1 \cap  I_2)  -\mutheta(I_1 \cap I_2)	\babs \Big] =o(1).
	\end{align*}
	\normalsize
	Therefore, we can assume without loss of generality that $\tilde \B$ is closed under finite unions.
	Let $G \subset \X$ be open and let $A_i \in \tilde \B$ be such that $G= \cup_{i=1}^\infty A_i.$
	Due to continuity of measure, for each $\delta>0$ there exists an integer $\tilde N$ such that $\mu(G)\leq \mu(\cup_{i=1}^{\tilde N} A_i) +\delta. $ 
	By the uniform absolute continuity of  $(\mutheta)_{\vartheta\in\Theta}$ over $\Theta $ with respect to $\mu$ there exists  an integer $N$ such that for any $\varepsilon>0$ 
	$$		\mutheta(G) \leq \mutheta(\cup_{i=1}^N A_i) +\varepsilon ,\quad \forall \vartheta \in \Theta.	$$
	Therefore,
	\begin{align*}
	\liminf_n \inf_{\vartheta \in \Theta} (	\muntheta(G) - \mutheta(G) + \varepsilon		)
	&\geq \liminf_n \inf_{\vartheta \in \Theta} \bbauf	\muntheta(G) - \mutheta(\cup_{i=1}^N A_i)		\bbzu \\
	&\geq \liminf_n \inf_{\vartheta \in \Theta} \bbauf	\muntheta(\cup_{i=1}^N A_i) - \mutheta(\cup_{i=1}^N A_i)		\bbzu=0,
	\end{align*}
	where the last equation is due to (\ref{defi_unif_conv_on_generator}), since $\cup_{i=1}^N A_i \in \tilde \B.$
	Letting $\varepsilon \to 0$ completes the proof by means of 3.\ of Theorem \ref{lemma:uniform_portmanteau}.
\end{proof}

%%%%%%%%%%%%%%%%%%%%%%%%%%%%%%%%%%%%%%%%%%%%%%%%%%%%%%%%%%%%%%%%%%%%%%%%%%%%%%%%%%%%%%%%%%%%%%%%%%%%%%%%%%%
%%%%%%%%%%%%%%%%%%%%%%%%%%%%%%%%%%%%%%%%%%%%%%%%%%%%%%%%%%%%%%%%%%%%%%%%%%%%%%%%%%%%%%%%%%%%%%%%%%%%%%%%%%%
%%%%%%%%%%%%%%%%%%%%%%%%%%%%%%%%%%%%%%%%%%%%%%%%%%%%%%%%%%%%%%%%%%%%%%%%%%%%%%%%%%%%%%%%%%%%%%%%%%%%%%%%%%%
%%%%%%%%%%%%%%%%%%%%%%%%%%%%%%%%%%%%%%%%%%%%%%%%%%%%%%%%%%%%%%%%%%%%%%%%%%%%%%%%%%%%%%%%%%%%%%%%%%%%%%%%%%%
%	 Proof of equivalence weak convergence and distribution convergence
%%%%%%%%%%%%%%%%%%%%%%%%%%%%%%%%%%%%%%%%%%%%%%%%%%%%%%%%%%%%%%%%%%%%%%%%%%%%%%%%%%%%%%%%%%%%%%%%%%%%%%%%%%%
%%%%%%%%%%%%%%%%%%%%%%%%%%%%%%%%%%%%%%%%%%%%%%%%%%%%%%%%%%%%%%%%%%%%%%%%%%%%%%%%%%%%%%%%%%%%%%%%%%%%%%%%%%%
%%%%%%%%%%%%%%%%%%%%%%%%%%%%%%%%%%%%%%%%%%%%%%%%%%%%%%%%%%%%%%%%%%%%%%%%%%%%%%%%%%%%%%%%%%%%%%%%%%%%%%%%%%%
%%%%%%%%%%%%%%%%%%%%%%%%%%%%%%%%%%%%%%%%%%%%%%%%%%%%%%%%%%%%%%%%%%%%%%%%%%%%%%%%%%%%%%%%%%%%%%%%%%%%%%%%%%%

\section{Relation between uniform weak convergence and uniform convergence in distribution} \label{sec_relation}

The following theorem relates the uniform weak convergence with the uniform convergence in distribution.

\begin{satz} \label{theorem_equivalence_weak_distrib_conv}
	Let $(\mutheta)_{\vartheta\in\Theta}$ be  uniformly absolutely continuous over $\Theta$ with respect to some continuous probability measure $\mu.$ 
	Let $\X=\R^d$ and thus $\A$ be the Borel-$\sigma$-algebra on $\R^d.$
	Setting 
	$$	F_{\vartheta,n}(\pvec{x}) = \muntheta\bbauf(-\infty,\pvec{x}]\bbzu, \quad \mbox{and} \quad F_{\vartheta}(\pvec{x}) = \mutheta\bbauf(-\infty,\pvec{x}]\bbzu, \quad \pvec{x}\in \R^d.	$$
	The following two statements are equivalent:
	\begin{enumerate}
		\item $ \muntheta \unifweakTheta \mutheta		$;
		\item $  \sup_{\vartheta \in \Theta} |F_{\vartheta,n}(\pvec{x}) - F_{\vartheta}(\pvec{x})| =o(1),  \quad \forall\pvec{x} \in \R^d.		$
	\end{enumerate}
\end{satz}
The latter theorem implies immediately the following corollary.

\begin{kor}\label{corollary_equivalence_weak_distrib_conv}
	If $(\Pxtheta)_{\vartheta \in \Theta}$ is uniformly absolutely continuous over $\Theta$ with respect to some continuous probability measure $Q,$ then	the following two statements are equivalent:
	\begin{enumerate}
		\item $  \Pxntheta \unifweakTheta \Pxtheta		$;
		\item $  \Xntheta \unifdistTheta \Xtheta. $
		
	\end{enumerate}
\end{kor}

\begin{proof} [\sl Proof of Theorem \ref{theorem_equivalence_weak_distrib_conv}]
	That 1.\ implies 2.\ follows immediately from 7.\ of Theorem \ref{lemma:uniform_portmanteau}, since the sets 
	$$	A_\pvec{x}=(-\infty,\pvec{x}], \quad \pvec{x}\in \R^d		$$
	are such that $\sup_{\vartheta \in \Theta} \mutheta(\partial A_\pvec{x}) =0,$ due to continuity of $(\mutheta)_{\vartheta\in\Theta}.$
	Thus, 
	$$	\sup_{\vartheta \in \Theta} \babs F_{\vartheta,n}(\pvec{x}) -  F_{\vartheta}(\pvec{x} ) \babs = 	\sup_{\vartheta \in \Theta} \babs \muntheta(A_\pvec{x}) -  \mutheta(A_\pvec{x}) \babs =o(1).	$$
	To verify that 2.\ implies 1.\ we make use of Lemma \ref{lemma_generator_equivalence} by considering 
	$$	\tilde \B = \{	(\pvec{a},\pvec{b}) \ | \ \pvec{a}<\pvec{b}, \, \pvec{a},\pvec{b}\in \R^d		\}		$$
	which satisfies the assumptions of Lemma \ref{lemma_generator_equivalence} if $\A$ is the Borel-$\sigma$-algebra on $\R^d.$
	Note that 
	$$	\muntheta\bbauf(\pvec{a},\pvec{b})\bbzu= F_{\vartheta,n}(\pvec{b}-) - F_{\vartheta,n}(\pvec{a}), \quad \mbox{and} \quad \mutheta\bbauf(\pvec{a},\pvec{b})\bbzu= F_{\vartheta}(\pvec{b}) - 		F_{\vartheta}(\pvec{a}),	$$
	where $F_{\vartheta,n}(\pvec{b}-) = \lim_{\pvec{x} \nearrow \pvec{b}} F_{\vartheta,n}(\pvec{x}).$
	Thus, it suffices to show that 
	$$	\sup_{\vartheta \in \Theta} \babs F_{\vartheta,n}(\pvec{x}-) -  F_{\vartheta}(\pvec{x} ) \babs=o(1), \quad \forall \pvec{x} \in \R^d.		$$	
	For this purpose, firstly obtain by monotonicity
	\begin{align*}
	\limsup_n \sup_{\vartheta \in \Theta} \bbauf F_{\vartheta,n}(\pvec{x}-) -  F_{\vartheta}(\pvec{x} )\bbzu
	&\leq \limsup_n \sup_{\vartheta \in \Theta} \bbauf F_{\vartheta,n}(\pvec{x}) -  F_{\vartheta}(\pvec{x} )\bbzu=0.
	\end{align*}
	Secondly, we have 
	\begin{align*}
	\liminf_n \inf_{\vartheta \in \Theta} \bbauf F_{\vartheta,n}(\pvec{x}-) -  F_{\vartheta}(\pvec{x} )\bbzu \geq 0,
	\end{align*}
	which concludes the proof.
	Indeed, let $\varepsilon>0$ then there exists by the  uniform absolute continuity of $(\mutheta)_{\vartheta\in\Theta}$ over $\Theta$ w.r.t\ $\mu$ a  $\delta>0$ such that
	$$	F_{\vartheta}\bbauf\pvec{x}-\delta(1,\ldots,1)^T \bbzu \geq F_{\vartheta}(\pvec{x}) - \varepsilon, \quad \forall \vartheta \in \Theta,	$$
	due to continuity of $\mu.$
	Thus by 2.\ there exists $N\in \N$ such that
	$ F_{\vartheta,n}\bbauf\pvec{x}-\delta(1,\ldots,1)^T\bbzu \geq F_{\vartheta}(\pvec{x}) - 2 \varepsilon$ for any $n\geq N$ and any $\vartheta \in \Theta.$
	By monotonicity obtain
	\begin{align*}
	\liminf_n \inf_{\vartheta \in \Theta} \bbauf F_{\vartheta,n}(\pvec{x}-) -  F_{\vartheta}(\pvec{x} ) \bbzu
	\geq - 2 \varepsilon,
	\end{align*}
	which yields the assertion by considering $\varepsilon \to 0.$
\end{proof}

%%%%%%%%%%%%%%%%%%%%%%%%%%%%%%%%%%%%%%%%%%%%%%%%%%%%%%%%%%%%%%%%%%%%%%%%%%%%%%%%%%%%%%%%%%%%%%%%%%%%%%%%%%%
%%%%%%%%%%%%%%%%%%%%%%%%%%%%%%%%%%%%%%%%%%%%%%%%%%%%%%%%%%%%%%%%%%%%%%%%%%%%%%%%%%%%%%%%%%%%%%%%%%%%%%%%%%%
%%%%%%%%%%%%%%%%%%%%%%%%%%%%%%%%%%%%%%%%%%%%%%%%%%%%%%%%%%%%%%%%%%%%%%%%%%%%%%%%%%%%%%%%%%%%%%%%%%%%%%%%%%%
%%%%%%%%%%%%%%%%%%%%%%%%%%%%%%%%%%%%%%%%%%%%%%%%%%%%%%%%%%%%%%%%%%%%%%%%%%%%%%%%%%%%%%%%%%%%%%%%%%%%%%%%%%%
%	 Approximation theorem
%%%%%%%%%%%%%%%%%%%%%%%%%%%%%%%%%%%%%%%%%%%%%%%%%%%%%%%%%%%%%%%%%%%%%%%%%%%%%%%%%%%%%%%%%%%%%%%%%%%%%%%%%%%
%%%%%%%%%%%%%%%%%%%%%%%%%%%%%%%%%%%%%%%%%%%%%%%%%%%%%%%%%%%%%%%%%%%%%%%%%%%%%%%%%%%%%%%%%%%%%%%%%%%%%%%%%%%
%%%%%%%%%%%%%%%%%%%%%%%%%%%%%%%%%%%%%%%%%%%%%%%%%%%%%%%%%%%%%%%%%%%%%%%%%%%%%%%%%%%%%%%%%%%%%%%%%%%%%%%%%%%
%%%%%%%%%%%%%%%%%%%%%%%%%%%%%%%%%%%%%%%%%%%%%%%%%%%%%%%%%%%%%%%%%%%%%%%%%%%%%%%%%%%%%%%%%%%%%%%%%%%%%%%%%%%

\smalltitle{Approximation result}

The next theorem is most useful to verify uniform convergence in distribution for two sequences of random vectors if the uniform distributional limit for one sequence is known and the distance between the random vectors is uniformly tending to zero in probability.

\begin{satz} \label{theorem_uniform_approximation}
	Let for any $\vartheta \in \Theta,$ $(\Yntheta)_{n\in\N}$   be a sequence of real-valued random vectors in $\R^d.$ 
	Suppose that $  \Xntheta \unifdistTheta \Xtheta  $ and $  \norm{ \Yntheta - \Xntheta} = o_{P,\Theta}(1)$ and in addition  $(\Pxtheta)_{\vartheta \in \Theta}$ is uniformly absolutely continuous over $\Theta$ with respect to some continuous probability measure $Q.$
	Then,
	$$	 \Yntheta \unifdistTheta \Xtheta.	$$
\end{satz}

\begin{proof} [\sl Proof of Theorem \ref{theorem_uniform_approximation}]
	Let $g:\R^d \to \R$ be a bounded, Lipschitz-continuous function with Lipschitz constant $L>0.$
	Let $\varepsilon>0,$ then
	\begin{align*}
	\sup_{\vartheta \in \Theta} \babs	\int g \, \d\Pyntheta -\int g \, \d\Pxntheta	\babs
	\leq L \, \varepsilon + 2 \norminf{g} \, \sup_{\vartheta \in \Theta} P_\vartheta\bbauf \norm{\Xntheta-\Yntheta}>\varepsilon \bbzu.
	\end{align*}
	Thus,
	\small
	\begin{align*}
	\sup_{\vartheta \in \Theta} \babs	\int g \, \d\Pyntheta &-\int g \, \d\Pxtheta	\babs \\
	&\leq L \, \varepsilon + 2 \, \norminf{g} \, \sup_{\vartheta \in \Theta} P_\vartheta\bbauf \norm{\Xntheta-\Yntheta}>\varepsilon \bbzu + \sup_{\vartheta \in \Theta} \babs	\int g \, \d\Pxntheta -\int g \, \d\Pxtheta	\babs.
	\end{align*}
	\normalsize
	The second term on the right-hand side of the latter inequality tends to zero, as does the third by 2.\ of Theorem \ref{lemma:uniform_portmanteau}.
	Considering $\varepsilon \to 0$ yields that $\sup_{\vartheta \in \Theta} |	\int g \, d\Pyntheta -\int g \, \d\Pxtheta	|=o(1).$
	Now, Corollary \ref{corollary_equivalence_weak_distrib_conv} and 2.\ of Theorem \ref{lemma:uniform_portmanteau} complete the proof.
\end{proof}

%%%%%%%%%%%%%%%%%%%%%%%%%%%%%%%%%%%%%%%%%%%%%%%%%%%%%%%%%%%%%%%%%%%%%%%%%%%%%%%%%%%%%%%%%%%%%%%%%%%%%%%%%%%
%%%%%%%%%%%%%%%%%%%%%%%%%%%%%%%%%%%%%%%%%%%%%%%%%%%%%%%%%%%%%%%%%%%%%%%%%%%%%%%%%%%%%%%%%%%%%%%%%%%%%%%%%%%
%%%%%%%%%%%%%%%%%%%%%%%%%%%%%%%%%%%%%%%%%%%%%%%%%%%%%%%%%%%%%%%%%%%%%%%%%%%%%%%%%%%%%%%%%%%%%%%%%%%%%%%%%%%
%%%%%%%%%%%%%%%%%%%%%%%%%%%%%%%%%%%%%%%%%%%%%%%%%%%%%%%%%%%%%%%%%%%%%%%%%%%%%%%%%%%%%%%%%%%%%%%%%%%%%%%%%%%
%	Auxiliary result for Levy continuity 
%%%%%%%%%%%%%%%%%%%%%%%%%%%%%%%%%%%%%%%%%%%%%%%%%%%%%%%%%%%%%%%%%%%%%%%%%%%%%%%%%%%%%%%%%%%%%%%%%%%%%%%%%%%
%%%%%%%%%%%%%%%%%%%%%%%%%%%%%%%%%%%%%%%%%%%%%%%%%%%%%%%%%%%%%%%%%%%%%%%%%%%%%%%%%%%%%%%%%%%%%%%%%%%%%%%%%%%
%%%%%%%%%%%%%%%%%%%%%%%%%%%%%%%%%%%%%%%%%%%%%%%%%%%%%%%%%%%%%%%%%%%%%%%%%%%%%%%%%%%%%%%%%%%%%%%%%%%%%%%%%%%
%%%%%%%%%%%%%%%%%%%%%%%%%%%%%%%%%%%%%%%%%%%%%%%%%%%%%%%%%%%%%%%%%%%%%%%%%%%%%%%%%%%%%%%%%%%%%%%%%%%%%%%%%%%

\section{Uniform versions of L\'evy's Continuity Theorem and Cram\'er-Wold Theorem} \label{sec_levy_cramer_wold}

We start with an auxiliary result giving a sufficient condition to verify weak convergence for a family of probability measures of random variables.

\begin{lemma} \label{lemma_auxiliary_equivalence_weak_cf}
	Suppose $(\Pxtheta)_{\vartheta \in \Theta}$ is uniformly absolutely continuous over $\Theta$ with respect to some continuous probability measure $Q.$
	Let $Y$ be some random vector in $\R^d.$
	If 
	$$  P_{\Xntheta+\sigma Y} \unifweakTheta P_{\Xtheta+\sigma Y}, \quad \forall \sigma>0		$$
	then
	$  \Pxntheta \unifweakTheta \Pxtheta.		$
\end{lemma}
\begin{proof} [\sl Proof of Lemma \ref{lemma_auxiliary_equivalence_weak_cf}]
	Let $g:\R^d \to \R$ be a bounded and continuous function.
	For all $\varepsilon>0$ there exists a $\delta>0$ such that if $  \norm{\pvec{x}-\pvec{z}}<\delta$ then $|g(\pvec{x})-g(\pvec{z})|<\varepsilon/6.$
	Let $\sigma$ be small enough such that $ \sup_{  \vartheta \in \Theta } \Pxtheta \bbauf \norm{Y}  \geq \delta \sigma\inv \bbzu < \nicefrac{\varepsilon}{12 \norminf{g}}.$
	Such a $\sigma$ exists due to the uniform absolute continuity of $\Xntheta$ w.r.t. $Q.$ 
	Thus,
	\begin{align*}
	\sup_{\vartheta \in \Theta} \babs \int g \, \d \Pxntheta - \int g \, \d  P_{\Xntheta+\sigma Y}  \babs
	&\leq \sup_{\vartheta \in \Theta} \babs \int 1_{ \{\sigma \norm{Y} < \delta\}} \, g \, \d (\Pxntheta -    P_{\Xntheta+\sigma Y})  \babs \\
	&\quad + \sup_{\vartheta \in \Theta}  \babs \int 1_{ \{\sigma \norm{Y} \geq \delta\}} \, g  \, \d\ (\Pxntheta -    P_{\Xntheta+\sigma Y}) \babs \\
	&\leq \varepsilon/6 + 2 \, \norminf{g}\ \sup_{  \vartheta \in \Theta } \Pxtheta \bbauf \norm{Y}  \geq \delta \sigma\inv\bbzu \leq \varepsilon/3
	\end{align*}
	Similarly, 
	$$		\sup_{\vartheta \in \Theta} \babs \int g \, \d \Pxtheta - \int g \, \d  P_{\Xtheta+\sigma Y}  \babs \leq \varepsilon/3.	$$
	Next, by assumption there exists $N\in \N$ such that
	$$		\sup_{\vartheta \in \Theta} \babs \int g \,\d P_{\Xntheta+\sigma Y} - \int g \,\d P_{\Xtheta+\sigma Y} \babs \leq \varepsilon/3, \quad \forall n \geq N.	$$
	Thus, by triangle inequality 
	\begin{align*}
	\sup_{\vartheta \in \Theta} \babs \int g \, \d \Pxntheta - \int g \, \d \Pxtheta \babs
	&\leq \varepsilon, \quad \forall n \geq N,
	\end{align*}
	which concludes the lemma.
\end{proof}

%%%%%%%%%%%%%%%%%%%%%%%%%%%%%%%%%%%%%%%%%%%%%%%%%%%%%%%%%%%%%%%%%%%%%%%%%%%%%%%%%%%%%%%%%%%%%%%%%%%%%%%%%%%
%%%%%%%%%%%%%%%%%%%%%%%%%%%%%%%%%%%%%%%%%%%%%%%%%%%%%%%%%%%%%%%%%%%%%%%%%%%%%%%%%%%%%%%%%%%%%%%%%%%%%%%%%%%
%%%%%%%%%%%%%%%%%%%%%%%%%%%%%%%%%%%%%%%%%%%%%%%%%%%%%%%%%%%%%%%%%%%%%%%%%%%%%%%%%%%%%%%%%%%%%%%%%%%%%%%%%%%
%%%%%%%%%%%%%%%%%%%%%%%%%%%%%%%%%%%%%%%%%%%%%%%%%%%%%%%%%%%%%%%%%%%%%%%%%%%%%%%%%%%%%%%%%%%%%%%%%%%%%%%%%%%
%	 Levy continuity theorem
%%%%%%%%%%%%%%%%%%%%%%%%%%%%%%%%%%%%%%%%%%%%%%%%%%%%%%%%%%%%%%%%%%%%%%%%%%%%%%%%%%%%%%%%%%%%%%%%%%%%%%%%%%%
%%%%%%%%%%%%%%%%%%%%%%%%%%%%%%%%%%%%%%%%%%%%%%%%%%%%%%%%%%%%%%%%%%%%%%%%%%%%%%%%%%%%%%%%%%%%%%%%%%%%%%%%%%%
%%%%%%%%%%%%%%%%%%%%%%%%%%%%%%%%%%%%%%%%%%%%%%%%%%%%%%%%%%%%%%%%%%%%%%%%%%%%%%%%%%%%%%%%%%%%%%%%%%%%%%%%%%%
%%%%%%%%%%%%%%%%%%%%%%%%%%%%%%%%%%%%%%%%%%%%%%%%%%%%%%%%%%%%%%%%%%%%%%%%%%%%%%%%%%%%%%%%%%%%%%%%%%%%%%%%%%%

The next theorem is sufficient to derive a uniform version of L\'evy's continuity theorem.

\begin{satz} \label{theorem_equivalence_weak_cf}
	Assume that $(\Pxtheta)_{\vartheta \in \Theta}$ is uniformly absolutely continuous over $\Theta$ with respect to some continuous probability measure $Q.$
	The following two statements are equivalent:
	\begin{enumerate}
		\item $  \Pxntheta \unifweakTheta \Pxtheta		$;
		\item $	 \sup_{\vartheta\in \Theta} |	\phi_{\Xntheta}(\pvec{t}) - \phi_{\Xtheta}(\pvec{t})	|=o(1), \quad \forall \pvec{t} \in \R^d.		$
		
	\end{enumerate}
\end{satz}

Corollary \ref{corollary_equivalence_weak_distrib_conv} and Theorem \ref{theorem_equivalence_weak_cf} imply the following uniform version of L\'evy's continuity theorem.

\begin{satz}[Uniform L\'evy's continuity theorem] \label{theorem:uniform_levy_continuiuty}
	Assume $(\Pxtheta)_{\vartheta \in \Theta}$ is uniformly absolutely continuous over $\Theta$ with respect to some continuous probability measure $Q.$
	Then the following two statements are equivalent:
	\begin{enumerate}
		\item $  \Xntheta \unifdistTheta \Xtheta  $;
		\item $	 \sup_{\vartheta\in \Theta} |	\phi_{\Xntheta}(\pvec{t}) - \phi_{\Xtheta}(\pvec{t})	|=o(1), \quad \forall \pvec{t} \in \R^d.		$
	\end{enumerate}
\end{satz}

\begin{proof} [\sl Proof of Theorem \ref{theorem_equivalence_weak_cf}]
	Suppose $  \Pxntheta \unifweakTheta \Pxtheta$ holds.
	The implication $$\sup_{\vartheta\in \Theta} |	\phi_{\Xntheta}(\pvec{t}) - \phi_{\Xtheta}(\pvec{t})	|=o(1), \quad \forall \pvec{t} \in \R^d		$$ follows immediately by splitting $\exp(i\pvec{t}^T X)$ into its real and imaginary part, which are both bounded continuous functions.

	Otherwise, assume that $	 \sup_{\vartheta\in \Theta} |	\phi_{\Xntheta}(\pvec{t}) - \phi_{\Xtheta}(\pvec{t})	|=o(1),$  for any $\pvec{t} \in \R^d.$
	Let $\sigma>0$ and let $Y \sim N_d(0,I_d)$ be independent of $\Xntheta$ for any $\vartheta\in \Theta$ and any $n \in \N.$
	Then, for any $g:\R^d \to \R$ which is bounded and continuous obtain with Fubini's theorem
	\begin{align*}
	\int g \,\d P_{\Xntheta + \sigma Y} 
	&= \int \int g( \pvec{y})  \varphi_{d,\sigma^2 I_d}(\pvec{y}-\pvec{x}) \, \d \pvec{y} \,\d P_{\Xntheta}(\pvec{x}) \\
	&=: \int g( \pvec{y})  I_n^\vartheta ( \pvec{y}) \, \d \pvec{y},
	\end{align*}
	where $\varphi_{d,\sigma^2 I_d}$ is the density function of $N_d(0,\sigma^2 I_d)$ and $I_n^\vartheta ( \pvec{y}) = \int \varphi_{d,\sigma^2 I_d}(\pvec{y}-\pvec{x}) \,\d P_{\Xntheta}(\pvec{x}).$
	Similarly, $$ \int g \,\d P_{\Xtheta + \sigma Y} = \int g( \pvec{y})  I^\vartheta ( \pvec{y}) \, \d \pvec{y} $$ with 
	$$I^\vartheta ( \pvec{y})= \int \varphi_{d,\sigma^2 I_d}(\pvec{y}-\pvec{x}) \,\d P_{\Xtheta}(\pvec{x}).$$
	By means of  the inversion formula and a substitution
	$$ \varphi_{d,\sigma^2 I_d}(\pvec{y}-\pvec{x}) = (2\pi\sigma)^{-d} \int \exp( i \sigma\inv \pvec{t}^T(\pvec{y} -\pvec{x}) - \nicefrac{\norm{\pvec{t}}^2}{2}  ) \,\d \pvec{t}.	$$
	Thus, by Fubini's theorem 
	\begin{align*}
	I_n^\vartheta ( \pvec{y}) 
	&=	(2\pi\sigma)^{-d}  \int \exp( - i \sigma\inv \pvec{t}^T\pvec{y} - \nicefrac{\norm{\pvec{t}}^2}{2}  ) \int \exp( i \sigma\inv \pvec{t}^T \pvec{x} ) \,\d P_{\Xntheta}(\pvec{x})    \, \d \pvec{t} \\
	&= 	(2\pi\sigma)^{-d}   \int \exp( - i \sigma\inv \pvec{t}^T\pvec{y} - \nicefrac{\norm{\pvec{t}}^2}{2}  ) \,  \phi_{\Xntheta}(\sigma\inv \pvec{t})    \, \d \pvec{t}
	\end{align*}
	and analogously $$  I^\vartheta( \pvec{y}) =  (2\pi\sigma)^{-d}   \int \exp( - i \sigma\inv \pvec{t}^T\pvec{y} - \nicefrac{\norm{\pvec{t}}^2}{2}  )\, \phi_{\Xtheta}(\sigma\inv \pvec{t})    \, \d \pvec{t}. $$
	Hence,
	\small
	\begin{align*}
	\sup_{\vartheta \in \Theta}&\babs\int g \,\d P_{\Xntheta + \sigma Y}  - 	\int g \,\d P_{\Xtheta + \sigma Y} 	\babs \\
	&\leq (2\pi\sigma)^{-d} \norminf{g} \int  \int |\exp( - i \sigma\inv \pvec{t}^T\pvec{y} - \nicefrac{\norm{\pvec{t}}^2}{2}  )| \, 	\sup_{\vartheta \in \Theta} \big| \phi_{\Xntheta}(\sigma\inv \pvec{t}) - \phi_{\Xtheta}(\sigma\inv \pvec{t}) \big|   \,\d \pvec{t} \, \d \pvec{y},
	\end{align*}
	\normalsize
	which tends to zero for $n \to \infty$ by dominated convergence.
	This implies $$  P_{\Xntheta+\sigma Y} \unifweakTheta P_{\Xtheta+\sigma Y}, \quad \forall \sigma>0		$$
	and the proof is finished by using Lemma \ref{lemma_auxiliary_equivalence_weak_cf}.
\end{proof}

%\section{Uniform Cram\'er-Wold Theorem}

The following uniform version of the Cram\'er-Wold Theorem follows immediately from Theorem \ref{theorem:uniform_levy_continuiuty}. 
\begin{satz} [Cram\'er-Wold Theorem] \label{theorem:unif_cramer_wold}
	Assume $(\Pxtheta)_{\vartheta \in \Theta}$ is uniformly absolutely continuous over $\Theta$ with respect to some continuous probability measure $Q.$
	If $ \pvec{a}^T \Xntheta \unifdistTheta \pvec{a}^T\Xtheta $ for any $\pvec{a} \in \R^d,$ then
	$\Xntheta \unifdistTheta \Xtheta.$
\end{satz}

%%%%%%%%%%%%%%%%%%%%%%%%%%%%%%%%%%%%%%%%%%%%%%%%%%%%%%%%%%%%%%%%%%%%%%%%%%%%%%%%%%%%%%%%%%%%%%%%%%%%%%%%%%%
%%%%%%%%%%%%%%%%%%%%%%%%%%%%%%%%%%%%%%%%%%%%%%%%%%%%%%%%%%%%%%%%%%%%%%%%%%%%%%%%%%%%%%%%%%%%%%%%%%%%%%%%%%%
%%%%%%%%%%%%%%%%%%%%%%%%%%%%%%%%%%%%%%%%%%%%%%%%%%%%%%%%%%%%%%%%%%%%%%%%%%%%%%%%%%%%%%%%%%%%%%%%%%%%%%%%%%%
%%%%%%%%%%%%%%%%%%%%%%%%%%%%%%%%%%%%%%%%%%%%%%%%%%%%%%%%%%%%%%%%%%%%%%%%%%%%%%%%%%%%%%%%%%%%%%%%%%%%%%%%%%%
%	 Continuous mapping
%%%%%%%%%%%%%%%%%%%%%%%%%%%%%%%%%%%%%%%%%%%%%%%%%%%%%%%%%%%%%%%%%%%%%%%%%%%%%%%%%%%%%%%%%%%%%%%%%%%%%%%%%%%
%%%%%%%%%%%%%%%%%%%%%%%%%%%%%%%%%%%%%%%%%%%%%%%%%%%%%%%%%%%%%%%%%%%%%%%%%%%%%%%%%%%%%%%%%%%%%%%%%%%%%%%%%%%
%%%%%%%%%%%%%%%%%%%%%%%%%%%%%%%%%%%%%%%%%%%%%%%%%%%%%%%%%%%%%%%%%%%%%%%%%%%%%%%%%%%%%%%%%%%%%%%%%%%%%%%%%%%
%%%%%%%%%%%%%%%%%%%%%%%%%%%%%%%%%%%%%%%%%%%%%%%%%%%%%%%%%%%%%%%%%%%%%%%%%%%%%%%%%%%%%%%%%%%%%%%%%%%%%%%%%%%

\section{Uniform versions of the Continuous Mapping Theorem and Slutzky's Theorem} \label{sec_cmt_slutzky}

\begin{lemma} \label{lemma_auxiliary_slutzky}
	Let for any $\vartheta \in \Theta,$ $(\Yntheta)_{n\in\N}$   be a sequence of real-valued random vectors in $\R^d$ and $(\pvec{c}^\vartheta)_{\vartheta \in \Theta}$ be deterministic real vectors in $\R^d$ with $\Yntheta = \pvec{c}^\vartheta + o_{P,\Theta}(1).$
	Furthermore, 	suppose $(\Pxtheta)_{\vartheta \in \Theta}$ is uniformly absolutely continuous over $\Theta$ with respect to some continuous probability measure $Q,$ and $  \Xntheta \unifdistTheta \Xtheta .$
	Then, it holds that
	$$	 (\Xntheta,\Yntheta) \unifdistTheta (\Xtheta,\pvec{c^\vartheta}).	$$
\end{lemma}

\begin{proof} [\sl Proof of Lemma \ref{lemma_auxiliary_slutzky}]
	For any $\pvec{x}_1,\pvec{x}_2\in \R^d$ it holds that 
	$$	F_{(\Xntheta,\pvec{c^\vartheta})}	(\pvec{x}_1,\pvec{x}_2) = F_{\Xntheta}	(\pvec{x}_1) \, F_{\pvec{c^\vartheta}}	(\pvec{x}_2), \quad \mbox{and} \quad F_{(\Xtheta,\pvec{c^\vartheta})}	(\pvec{x}_1,\pvec{x}_2) = F_{\Xtheta}	(\pvec{x}_1) \, F_{\pvec{c^\vartheta}}	(\pvec{x}_2).	$$
	Thus,
	\begin{align*}
	\sup_{\vartheta \in \Theta} \babs	F_{(\Xntheta,\pvec{c^\vartheta})}	(\pvec{x}_1,\pvec{x}_2) - F_{(\Xtheta,\pvec{c^\vartheta})}	(\pvec{x}_1,\pvec{x}_2) \babs
	\leq \sup_{\vartheta \in \Theta} \babs	F_{\Xntheta}	(\pvec{x}_1) - F_{\Xtheta}	(\pvec{x}_1)\babs =o(1),
	\end{align*}
	which shows $	(\Xntheta,\pvec{c^\vartheta}) \unifdistTheta (\Xtheta,\pvec{c^\vartheta}).$
	Note that $ \sup_{\vartheta \in \Theta} \norm{ (\Xntheta,\Yntheta) - (\Xntheta,\pvec{c^\vartheta})  } = o_{P,\Theta}(1) $
	and with Theorem \ref{theorem_uniform_approximation} the assertion follows.
\end{proof}

The following theorem is a uniform version of the continuous mapping theorem.
\begin{satz} [Uniform Continuous Mapping Theorem] \label{theorem_uniform_CMT}
	Let $H:\R^d \to \R^s$ be continuous.
	If  $(\Pxtheta)_{\vartheta \in \Theta}$ is  uniformly absolutely continuous over $\Theta$ with respect to some continuous probability measure $Q$ and $  \Xntheta \unifdistTheta \Xtheta  ,$ then 
	$$		  H(\Xntheta) \unifdistTheta H(\Xtheta). $$
\end{satz}
The proof follows directly from Corollary \ref{corollary_equivalence_weak_distrib_conv} since for any bounded continuous function $g:\R^d \to \R$ the composition $g \circ H$ is still bounded and continuous.

%%%%%%%%%%%%%%%%%%%%%%%%%%%%%%%%%%%%%%%%%%%%%%%%%%%%%%%%%%%%%%%%%%%%%%%%%%%%%%%%%%%%%%%%%%%%%%%%%%%%%%%%%%%
%%%%%%%%%%%%%%%%%%%%%%%%%%%%%%%%%%%%%%%%%%%%%%%%%%%%%%%%%%%%%%%%%%%%%%%%%%%%%%%%%%%%%%%%%%%%%%%%%%%%%%%%%%%
%%%%%%%%%%%%%%%%%%%%%%%%%%%%%%%%%%%%%%%%%%%%%%%%%%%%%%%%%%%%%%%%%%%%%%%%%%%%%%%%%%%%%%%%%%%%%%%%%%%%%%%%%%%
%%%%%%%%%%%%%%%%%%%%%%%%%%%%%%%%%%%%%%%%%%%%%%%%%%%%%%%%%%%%%%%%%%%%%%%%%%%%%%%%%%%%%%%%%%%%%%%%%%%%%%%%%%%
%	 Slutzky's theorem
%%%%%%%%%%%%%%%%%%%%%%%%%%%%%%%%%%%%%%%%%%%%%%%%%%%%%%%%%%%%%%%%%%%%%%%%%%%%%%%%%%%%%%%%%%%%%%%%%%%%%%%%%%%
%%%%%%%%%%%%%%%%%%%%%%%%%%%%%%%%%%%%%%%%%%%%%%%%%%%%%%%%%%%%%%%%%%%%%%%%%%%%%%%%%%%%%%%%%%%%%%%%%%%%%%%%%%%
%%%%%%%%%%%%%%%%%%%%%%%%%%%%%%%%%%%%%%%%%%%%%%%%%%%%%%%%%%%%%%%%%%%%%%%%%%%%%%%%%%%%%%%%%%%%%%%%%%%%%%%%%%%
%%%%%%%%%%%%%%%%%%%%%%%%%%%%%%%%%%%%%%%%%%%%%%%%%%%%%%%%%%%%%%%%%%%%%%%%%%%%%%%%%%%%%%%%%%%%%%%%%%%%%%%%%%%

%\section{Uniform version of Slutzky's Theorem}

%
\begin{satz}[Uniform Slutzky's Theorem] \label{theorem_uniform_slutky}
	Let for any $\vartheta \in \Theta,$ $(\Yntheta)_{n\in\N}$   be a sequence of real-valued random vectors in $\R^d$ and $(\pvec{c}^\vartheta)_{\vartheta \in \Theta}$ be deterministic real vectors in $\R^d$ with $\Yntheta = \pvec{c}^\vartheta + o_{P,\Theta}(1).$
	Furthermore, 	suppose $(\Pxtheta)_{\vartheta \in \Theta}$ is uniformly absolutely continuous over $\Theta$ with respect to some continuous probability measure $Q,$ and $  \Xntheta \unifdistTheta \Xtheta  ,$ then
	$$	 \Xntheta + \Yntheta \unifdistTheta \Xtheta + \pvec{c}^\vartheta \quad \mbox{and}\quad	\Xntheta \cdot \Yntheta \unifdistTheta \Xtheta \cdot \pvec{c}^\vartheta,	$$
	where the multiplication is to be understood componentwise.
\end{satz}
\begin{proof} [\sl Proof of Theorem \ref{theorem_uniform_slutky}]
	The functions $(\pvec{x},\pvec{y}) \mapsto \pvec{x} + \pvec{y}$ and $(\pvec{x},\pvec{y}) \mapsto \pvec{x} \cdot \pvec{y}$ are continuous, such that
	applying the uniform continuous mapping theorem \ref{theorem_uniform_CMT} in combination with Lemma  \ref{lemma_auxiliary_slutzky} yields the assertion.
\end{proof}

%%%%%%%%%%%%%%%%%%%%%%%%%%%%%%%%%%%%%%%%%%%%%%%%%%%%%%%%%%%%%%%%%%%%%%%%%%%%%%%%%%%%%%%%%%%%%%%%%%%%%%%%%%%
%%%%%%%%%%%%%%%%%%%%%%%%%%%%%%%%%%%%%%%%%%%%%%%%%%%%%%%%%%%%%%%%%%%%%%%%%%%%%%%%%%%%%%%%%%%%%%%%%%%%%%%%%%%
%%%%%%%%%%%%%%%%%%%%%%%%%%%%%%%%%%%%%%%%%%%%%%%%%%%%%%%%%%%%%%%%%%%%%%%%%%%%%%%%%%%%%%%%%%%%%%%%%%%%%%%%%%%
%%%%%%%%%%%%%%%%%%%%%%%%%%%%%%%%%%%%%%%%%%%%%%%%%%%%%%%%%%%%%%%%%%%%%%%%%%%%%%%%%%%%%%%%%%%%%%%%%%%%%%%%%%%
%	 Lindeberg-Feller Theorem
%%%%%%%%%%%%%%%%%%%%%%%%%%%%%%%%%%%%%%%%%%%%%%%%%%%%%%%%%%%%%%%%%%%%%%%%%%%%%%%%%%%%%%%%%%%%%%%%%%%%%%%%%%%
%%%%%%%%%%%%%%%%%%%%%%%%%%%%%%%%%%%%%%%%%%%%%%%%%%%%%%%%%%%%%%%%%%%%%%%%%%%%%%%%%%%%%%%%%%%%%%%%%%%%%%%%%%%
%%%%%%%%%%%%%%%%%%%%%%%%%%%%%%%%%%%%%%%%%%%%%%%%%%%%%%%%%%%%%%%%%%%%%%%%%%%%%%%%%%%%%%%%%%%%%%%%%%%%%%%%%%%
%%%%%%%%%%%%%%%%%%%%%%%%%%%%%%%%%%%%%%%%%%%%%%%%%%%%%%%%%%%%%%%%%%%%%%%%%%%%%%%%%%%%%%%%%%%%%%%%%%%%%%%%%%%

\section{Uniform Lindeberg-Feller-Theorem} \label{sec_lindeberg_feller}

For the derivation of a uniform version of the Lindeberg-Feller-Theorem we need the following auxiliary results.

\begin{lemma} \label{lemma_durret_1}
	For $z_1,\ldots,z_n \in \C$ and $w_1,\ldots,w_n\in \C,$ where $\C$ is the field if complex numbers, with $\sup_{i} | \max \{z_i , w_i \} | \leq \theta$ for some $\theta>0.$
	Then,
	$$		\Big|	\prod_{i=1}^n z_i - \prod_{i=1}^n w_i 	\Big|	\leq \theta^{n-1} \sum_{i=1}^n |z_i - w_i|.	$$
\end{lemma}
For the proof see for instance Lemma 3.4.3 in \citet{durrett2010probability}.

\begin{lemma} \label{lemma_durret_2}
	Let $r\in \N.$ Then for any random variable $X$ with $\E|X|^{r+1} < \infty$ it holds
	$$	\Babs \,	\varphi_X(t)  - \sum_{k=0}^r \E \frac{(itX)^k}{k!}	\Babs \leq \E \min\{ |tX|^{r+1}, 2|tX|^r		\}	.	$$
\end{lemma}
A proof of this result is given for instance in \citet{durrett2010probability}, see Equation (3.3.3).

With L'H\^opital's rule obtain the following result.
\begin{lemma} \label{lemma_durret_3}
	Let $\vartheta \in \Theta$ and for each $n$ let $ c_{n,j}^\vartheta,$ $1\leq j \leq n$ be real-values with 
	\begin{enumerate}
		\item $\sup_{  \vartheta \in \Theta } \max_{j}  |c_{n,j}^\vartheta| \to 0;$
		\item $\sup_{  \vartheta \in \Theta } | \sum_{j=1}^n c_{n,j}^\vartheta	- \lambda	|$ for some $\lambda\in \R;$
		\item $\sup_{  \vartheta \in \Theta } \sup_{n} \sum_{j=1}^n |c_{n,j}^\vartheta| < \infty.$
	\end{enumerate}
	Then, $	\sup_{  \vartheta \in \Theta } | \prod_{j=1}^n (1+c_{n,j}^\vartheta)   - \exp(\lambda)	| \to 0.		$
\end{lemma}

\newcommand{\Xnitheta}{X_{n,i}^\vartheta}
\newcommand{\Pxitheta}{P_{X_{n,i}^\vartheta}}

\begin{satz} \label{theorem:unif_lindeberg_feller_multi}
	For each $n\in \N$ and $\vartheta\in \Theta$ let $\Xnitheta,$ $1\leq i \leq n$ be centered and independent random vectors in $\R^d.$
	Assume that $(\Pxitheta)_{\vartheta \in \Theta}$ is uniformly absolutely continuous over $\Theta$ with respect to some continuous probability measure $Q.$
	Moreover, suppose that
	\begin{enumerate}
		\item $	\sup_{\vartheta \in \Theta} \norm{   \sum_{i=0}^n \E \Xnitheta \big(\Xnitheta\big)^T- \Sigma	} =o(1),	$ for some semi-positive-definite matrix $\Sigma \in \R^{d\times d};$
		\item For any $\varepsilon>0$ it holds that $\limsup_n \sup_{  \vartheta \in \Theta } \sum_{i=1}^n  \E \big(	\norm{\Xnitheta}^2 \, 1_{\norm{\Xnitheta} > \varepsilon}			\big) =0;$
	\end{enumerate}
	Then, $$\sup_{\vartheta\in \Theta} \big| \, P\big( X_{n,1}^\vartheta + \ldots + X_{n,n}^\vartheta \leq \pvec{x}\big) - \Phi_{\Sigma}(\pvec{x})	\, 		\big|=o(1),			$$
	where $\Phi_{\Sigma}$ is the cumulative distribution function of $N(0,\Sigma).$ 
\end{satz}
%\\

\begin{proof}
	By the uniform Cram\'er-Wold theorem \ref{theorem:unif_cramer_wold} it suffices to show the univariate version of Theorem \ref{theorem:unif_lindeberg_feller_multi}.
	Therefore, suppose that $\Xnitheta$ are random variables in $\R$ and $\Sigma = \sigma^2$ is a positive value.
	We show that 
	\begin{align} \label{eq_help_lindeberg_feller_1}
	\sup_{\vartheta \in \Theta} \big| \prod_{i=1}^n	\varphi_{\Xnitheta}(t)	- \exp\big(		- \nicefrac{t^2 \sigma^2}{2}	\big)			\big| =o(1), \quad \forall t\in \R.
	\end{align}
	By the uniform L\'evy-continuity Theorem \ref{theorem:uniform_levy_continuiuty} this would complete the proof.
	\newcommand{\zni}{z_{n,i}^\vartheta}
	\newcommand{\wni}{w_{n,i}^\vartheta}
	
	Let $t\in \R$ be fixed. Set $\zni=\varphi_{\Xnitheta}(t)$ and $\wni=(1-\nicefrac{t^2 \E [(\Xnitheta)^2]}{2}).$
	At	first, note that $|\zni|\leq 1$ for each $n,i,\vartheta$ as well as 
	\begin{align*}
	\E [(\Xnitheta)^2] \leq \varepsilon^2 + \E \big[(\Xnitheta)^2 1_{|\Xnitheta|>\varepsilon}\big],
	\end{align*}
	for any $\varepsilon>0,$ so that by the second assumption 
	\begin{align} \label{eq_help_lindeberg_feller_2}
	\sup_{  \vartheta \in \Theta } \sup_{i} \E [(\Xnitheta)^2] \to 0.
	\end{align}
	This implies for $n$ large enough that $|\wni|\leq 1$ for all $n,i$ and $\vartheta.$
	Hence, by Lemma \ref{lemma_durret_1}
	%	\\
	\begin{align} \label{eq_help_lindeberg_feller_3}
	\Big|	\prod_{i=1}^n \zni - \prod_{i=1}^n \wni 	\Big|	\leq \sum_{i=1}^n |\zni - \wni|.		
	\end{align}	
	By Lemma \ref{lemma_durret_2} we have for any $\varepsilon>0$
	%	\\
	\begin{align*}
	|\zni - \wni| 
	&\leq \E \min\Big(	|t \, \Xnitheta|^3 \, , \,  2 |t \, \Xnitheta|^2		\Big) \\
	%		
	%		&\leq t^3\, \E \Big(	| \Xnitheta|^3 \, 1_{|\Xnitheta|\leq \varepsilon}\Big) +    2 t^2 \, \E\Big(  \Xnitheta|^2 \, _{|\Xnitheta|> \varepsilon}		\Big) \\
	%		
	&\leq \varepsilon \, t^3\, \E \Big(	| \Xnitheta|^2 \, 1_{|\Xnitheta|\leq \varepsilon}\Big) +    2 t^2 \, \E\Big(  |\Xnitheta|^2 \, _{|\Xnitheta|> \varepsilon}		\Big).
	\end{align*}
	By the assumptions 1.\ and 2.\ of the theorem in combination with (\ref{eq_help_lindeberg_feller_3}) it follows that
	\begin{align*}
	\limsup_n \sup_{  \vartheta \in \Theta } 	\Big|	\prod_{i=1}^n \zni - \prod_{i=1}^n \wni 	\Big|
	\leq \limsup_n \sup_{  \vartheta \in \Theta } 
	\sum_{i=1}^n |\zni - \wni| \leq \varepsilon \, t^3 \, \sigma^2.
	\end{align*}
	Letting $\varepsilon \to 0$ the right-hand side of the preceding display converges to zero.
	To verify (\ref{eq_help_lindeberg_feller_1}) it remains to show that
	\begin{align*} 
	%	\label{eq_help_lindeberg_feller_4}
	%		
	\limsup_n \sup_{  \vartheta \in \Theta } 	\Big|  \prod_{i=1}^n \wni - 	\exp\big(		- \nicefrac{t^2 \sigma^2}{2} \big)	\Big|= 0.	
	\end{align*}
	But this is immediate from Lemma \ref{lemma_durret_3} by setting $c_{n,i}^\vartheta = - \nicefrac{t^2 \E [(\Xnitheta)^2]}{2}.$
	Indeed, (\ref{eq_help_lindeberg_feller_2}) shows 1.\ of Lemma \ref{lemma_durret_3}, while assumption 2.\ of this theorem implies
	\begin{align*}
	\sup_{  \vartheta \in \Theta } \babs \sum_{i=1}^n c_{n,i}^\vartheta - \frac{\sigma^2 \, t^2}{2}		\babs \to 0,
	\end{align*}
	and assumption 2.\ shows the last assumption of Lemma \ref{lemma_durret_3}.
\end{proof}

\newpage
\bibliographystyle{chicago}
\bibliography{Literature_supplement}

\newpage

\end{document}